\documentclass[reqno]{amsart}
\title[A $q$-microscope for mock-theta patterns]{A $q$-microscope for mock-theta denominator patterns}
\usepackage{amssymb,amsmath,amsthm}
\usepackage[hidelinks]{hyperref}
\theoremstyle{definition}

\theoremstyle{plain}
\newtheorem{theorem}{Theorem}

\newtheorem{corollary}{Corollary}
\newtheorem{conjecture}[theorem]{Conjecture}

\theoremstyle{remark}

\newtheorem{remark}{Remark}

\newcommand{\Ck}{\mathcal{C}_k}
\newcommand{\Fk}{\mathcal{F}_k}
\newcommand{\Gk}{\mathcal{G}_k}
\newcommand{\Hk}{\mathcal{H}_k}
\newcommand{\T}{\mathcal{T}}
\newcommand{\Q}{\mathbb{Q}}
\newcommand{\fr}{\frac}

\numberwithin{equation}{section}

\begin{document}
\author[M. El Bachraoui]{Mohamed El Bachraoui}
\address{Dept. Math. Sci,
United Arab Emirates University, PO Box 15551, Al-Ain, UAE}
\email{melbachraoui@uaeu.ac.ae}


\begin{abstract}
We give a systematic creative-microscoping treatment of truncated basic hypergeometric sums attached to three classical third-order mock theta denominator patterns, represented by $\phi(q)$, $\omega(q)$, and $\nu(q)$. The main conceptual point is that these denominator patterns naturally admit exact finite theta evaluations at $a=q^{\pm n}$, microscopic $q$-supercongruences, and cyclotomic/$p$-adic consequences within a single framework. For
\[
S(q,a):=\sum_{n\ge 0}\frac{q^{2n+1}(aq,q/a;q^2)_n}{(-q^2;q^2)_n},
\]
we prove, for odd $n$, a finite theta evaluation at $a=q^{\pm n}$, equivalently a microscopic $q$-supercongruence. We prove an analogous finite theta evaluation for the $\nu$-type denominator
\[
N_m(q,a):=\sum_{j=0}^{m} q^{2j+1}\frac{(aq,q/a;q^2)_j}{(-q;q^2)_{j+1}}.
\]
We then introduce the one-parameter denominator family
\[
\T(q;b):=\sum_{n\ge 0} q^{2n}\frac{(q;q^2)_n^2}{(q^2;q^2)_n(b;q^2)_n},
\]
whose specializations include the $\omega$-type shifted denominator ($b=q^3$) and the two-color family ($b=q^{2k}$). For its truncated $a$-extension we obtain an exact evaluation at $a=q^{\pm n}$, deduce microscopic congruences modulo $(a-q^n)(1-aq^n)$, and derive cyclotomic consequences, including zero congruences in the $q^{2k}$-subfamily and nontrivial $p$-adic versions for the cases $b=q$ and $b=q^3$.
\end{abstract}

\date{\textit{\today}}
\maketitle

\section{Introduction}

The purpose of this paper is to show that several classical mock-theta denominator patterns admit a common creative-microscoping treatment. For truncated basic hypergeometric sums attached to the $\phi$-, $\omega$-, and $\nu$-type denominators, we obtain exact evaluations at the odd specializations $a=q^{\pm n}$ and deduce microscopic $q$-supercongruences together with cyclotomic and $p$-adic consequences. Thus the point of view is that the arithmetic structure is carried by the denominator patterns themselves, rather than by isolated congruence examples.

The present paper is organized around three classical third-order mock theta denominator patterns:
\[
(-q^2;q^2)_n,
\qquad
(q;q^2)_{n+1},
\qquad
(-q;q^2)_{n+1}.
\]
These are represented, respectively, by Ramanujan's function
\begin{equation}\label{eq:phi-intro}
\phi(q):=\sum_{n\ge 0}\frac{q^{n^2}}{(-q^2;q^2)_n},
\end{equation}
the alternative series
\begin{equation}\label{eq:omega-intro}
\omega(q):=\sum_{n\ge 0}\frac{q^n}{(q;q^2)_{n+1}},
\end{equation}
and Watson's function
\begin{equation}\label{eq:nu-intro}
\nu(q):=\sum_{n\ge 0}\frac{q^{n(n+1)}}{(-q;q^2)_{n+1}}.
\end{equation}
The same three functions also appear naturally in the partition-theoretic work of Andrews, Dixit and Yee~\cite{andrews-dixit-yee}.

Our starting point is the $q$-hypergeometric series
\begin{equation}\label{eq:S}
S(q):=\sum_{n\ge 0}\frac{q^{2n+1}(q;q^2)_n^2}{(-q^2;q^2)_n},
\end{equation}
which recently arose in~\cite{andrews-elbachraoui, banerjee-etal} in connection with two-color partitions. Although~\eqref{eq:S} was originally motivated there by $\omega(q)$, its denominator places it more naturally in the $\phi$-type category~\eqref{eq:phi-intro}.
In this paper the main results are grouped by denominator pattern and the two-color material is treated later as a specialization. This organization is meant to emphasize that the two-color congruences arise from a broader mock-theta denominator framework rather than as isolated examples.

For completeness, we note that the first two denominator patterns also occur in McIntosh's second-order mock theta functions. In particular,
\begin{equation}\label{eq:ABmu-short}
\mu(q)=\sum_{n\ge 0}\frac{(-1)^n q^{n^2}(q;q^2)_n}{(-q^2;q^2)_n^2}
\end{equation}
again involves the factor $(-q^2;q^2)_n$; see~\cite{mcintosh-second-order}. We mention this only as background, since the exposition below is organized around the third-order trio~\eqref{eq:phi-intro}--\eqref{eq:nu-intro}.

Throughout, $(a;q)_n:=\prod_{j=0}^{n-1}(1-aq^j)$ denotes the $q$-Pochhammer symbol and $(a_1,\ldots,a_r;q)_n:=\prod_{i=1}^r(a_i;q)_n$.
For nonnegative integer $m$ and $n$ we use ${n\brack m}_q$ for the $q$-binomial coefficient given by
\[
{n\brack m} = {n \brack m}_q
=\begin{cases}
\frac{(q;q)_n}{(q;q)_m (q;q)_{n-m}}\ \text{if\ } n\geq m, \\
0\ \text{otherwise.}
\end{cases}
\]
Let
$\mathbb{Z}[q]$ denote the set of polynomials in $q$ with integer coefficients and let
$\mathbb{Z}(q)$ denote the set of rational functions in $q$ with integer coefficients.
The $n$-th cyclotomic polynomial is the polynomial in $\mathbb{Z}[q]$ given by
\[
\Phi_n(q) = \prod_{\substack{j=1 \\ \gcd(j,n)=1}}^n(q- \zeta^j) ,
\]
where $\zeta = e^{2\pi i/n}$ is the $n$-th root of unity.
Given polynomials $A_1(q), A_2(q), P(q) \in \mathbb{Z}[q]$, the congruence
$\fr{A_1(q)}{A_2(q)} \equiv 0 \pmod{P(q)}$ means that $P(q)$ divides $A_1(q)$ and that
$\gcd\big(P(q),A_2(q) \big) =1$; and in general for rational functions $A(q), B(q) \in\mathbb{Z}(q)$,
the congruence $A(q)\equiv B(q) \pmod{P(q)}$ means that $A(q)-B(q)\equiv 0\pmod{P(q)}$.

A guiding method in the subject is \emph{creative microscoping}, introduced by Guo and Zudilin~\cite{guo-zudilin}, which turns congruences modulo powers of cyclotomic polynomials into parameterized identities at roots of unity.
Recent developments in this direction include Jackson's transformations, Gasper’s Karlsson--Minton type summations,
squares and higher-fold basic hypergeometric series, and higher-power cyclotomic congruences; see, for example, \cite{guo-li-squares,guo-schlosser-threefam,maithy-barman,wei-jcta}.

From a historical point of view, this emphasis on roots of unity is not accidental. In modern language, Ramanujan’s original definition is already $q$-microscopic in spirit: it is a root-by-root comparison between a mock theta function and a local theta correction. Indeed, in his last letter to Hardy, Ramanujan requires that for each root of unity $\xi$ there exist a modular form $M_{\xi}(q)$ and a rational number $\alpha$ such that $F(q)-q^{\alpha}M_{\xi}(q)$ remains bounded as $q\to\xi$ radially, while no single modular form works uniformly at all roots of unity; see~\cite[pp.~220--223]{berndt-rankin} and the modern formulations in~\cite{rhoades-definition,bringmann-rolen}. Conversely, creative microscoping is built precisely on evaluating truncated basic hypergeometric series at roots of unity, and Guo and Zudilin explicitly observed that this perspective is close in spirit to the study of mock theta functions and quantum modular forms~\cite[p.~4]{guo-zudilin}. It is therefore striking that these two literatures have so rarely met directly. To the best of our knowledge, the classical mock theta series themselves have not previously been treated in the language of creative microscoping or cyclotomic $q$-supercongruences in any systematic way. Existing arithmetic work on mock theta functions has concentrated mainly on coefficient congruences and related partition congruences~\cite{andersen,andrews-passary-sellers-yee,waldherr}, while the root-of-unity side has largely been pursued through radial limits and quantum or antiquantum $q$-series
identities~\cite{bringmann-rolen,folsom-metacarpa,folsom-metacarpa-anti,folsom-ono-rhoades}. This historical parallel motivates the viewpoint adopted here: the classical mock-theta denominator patterns themselves naturally support microscopic congruences and exact finite theta evaluations. More specifically, the $\phi$- and $\nu$-type cases lead to explicit finite theta formulas at $a=q^{\pm n}$, while the one-parameter family $\T(q;b)$ simultaneously captures the shifted $\omega$-type case and the two-color $q^{2k}$-family, leading in turn to cyclotomic and $p$-adic consequences. In this sense, the paper is intended as a systematic creative-microscoping treatment of classical mock-theta denominator patterns rather than merely a collection of separate congruences.


Define the two-parameter series
\begin{equation}\label{eq:Sa}
S(q,a):=\sum_{n\ge 0}\frac{q^{2n+1}(aq,q/a;q^2)_n}{(-q^2;q^2)_n}
\end{equation}
and, for odd $n=2m+1$,
\[
S_m(q,a):=\sum_{j=0}^{m}\frac{q^{2j+1}(aq,q/a;q^2)_j}{(-q^2;q^2)_j}.
\]
Our first result is a finite theta evaluation for the $\phi$-type denominator.

\begin{theorem}\label{thm:finite-theta}
Let $n=2m+1$ be odd. Then
\begin{equation}\label{eq:finite-theta}
\sum_{j=0}^{m}\frac{q^{2j+1}(q^{1-n},q^{n+1};q^2)_j}{(-q^2;q^2)_j}
= (-1)^m q^{2m(m+1)+1}\sum_{r=-m}^{m}(-1)^r q^{-2r^2}.
\end{equation}
\end{theorem}

\begin{corollary}\label{cor:microS}
Let $n=2m+1$ be odd. Then modulo $(a-q^n)(1-aq^n)$, we have
\begin{equation}\label{eq:micro-S}
\begin{aligned}
\sum_{j=0}^{m}\frac{q^{2j+1}(aq,q/a;q^2)_j}{(-q^2;q^2)_j}
&\equiv (-1)^m q^{2m(m+1)+1}\sum_{r=-m}^{m}(-1)^r q^{-2r^2}.\\
\end{aligned}
\end{equation}
\end{corollary}

The $\nu$-type analogue is obtained by replacing $(-q^2;q^2)_j$ with $(-q;q^2)_{j+1}$.

\begin{theorem}\label{thm:nu-finite-theta}
Let $n=2m+1$ be odd, and define
\[
N_m(q,a):=\sum_{j=0}^{m} q^{2j+1}\frac{(aq,q/a;q^2)_j}{(-q;q^2)_{j+1}}.
\]
Then
\begin{equation}\label{eq:nu-finite-theta}
(1+q^n)\sum_{j=0}^{m} q^{2j+1}
\frac{(q^{1-n},q^{n+1};q^2)_j}{(-q;q^2)_{j+1}}
=
(-1)^m q^{(m+1)(2m+1)}
\sum_{r=-m}^{m} (-1)^r q^{-2r^2-r}.
\end{equation}
\end{theorem}

\begin{corollary}\label{cor:nu-micro}
Let $n=2m+1$ be odd. Then modulo $(a-q^n)(1-aq^n)$, we have
\begin{equation}\label{eq:nu-micro}
(1+q^n)N_m(q,a)
\equiv
(-1)^m q^{(m+1)(2m+1)}
\sum_{r=-m}^{m} (-1)^r q^{-2r^2-r}.
\end{equation}
\end{corollary}

The second main object in the paper is the one-parameter denominator family
\begin{equation}\label{eq:T-infinite}
\T(q;b):=\sum_{n\ge 0} q^{2n}\frac{(q;q^2)_n^2}{(q^2;q^2)_n(b;q^2)_n}.
\end{equation}
This family contains, in particular, the shifted denominator $(q;q^2)_{n+1}$ occurring in~\eqref{eq:omega-intro} via the specialization $b=q^3$, and the two-color family of~\cite{andrews-elbachraoui} via $b=q^{2k}$.

For odd $n=2m+1$, define the truncated $a$-extension of~\eqref{eq:T-infinite} by
\begin{equation}\label{eq:Tmb}
\T_m(q,a;b):=\sum_{j=0}^{m} q^{2j}\frac{(aq,q/a;q^2)_j}{(q^2;q^2)_j(b;q^2)_j}.
\end{equation}
Our third theorem is an exact evaluation at $a=q^{\pm n}$.

\begin{theorem}\label{thm:free-b}
Let $n=2m+1$ be odd. Then
\begin{equation}\label{eq:Tmb-exact}
\T_m(q,q^{\pm n};b)=q^{2m(m+1)}\frac{(bq^{-n-1};q^2)_m}{(b;q^2)_m}.
\end{equation}
\end{theorem}

\begin{corollary}\label{cor:Tmb-micro}
Let $n=2m+1$ be odd. Then modulo $(a-q^n)(1-aq^n)$, we have
\begin{equation}\label{eq:Tmb-micro}
\T_m(q,a;b)\equiv q^{2m(m+1)}\frac{(bq^{-n-1};q^2)_m}{(b;q^2)_m}.
\end{equation}
\end{corollary}

\begin{remark}
Let $n=2m+1$ and define
\[
c_j(q,a;b):=q^{2j}\frac{(aq,q/a;q^2)_j}{(q^2;q^2)_j(b;q^2)_j}.
\]
Then
\[
\sum_{r=0}^{n-1}\sum_{j=0}^{r} c_j(q,q^{\pm n};b)c_{r-j}(q,q^{\pm n};b)
=
q^{4m(m+1)}\frac{(bq^{-n-1};q^2)_m^2}{(b;q^2)_m^2}.
\]
This is exactly the same convolution-to-square step used in
\cite[Lemma~1(a) and equation~(14)]{elbachraoui-squares}. Indeed,
$c_j(q,q^{\pm n};b)=0$ for $j>m$, so the left-hand side is
\[
\Bigl(\sum_{j=0}^{m} c_j(q,q^{\pm n};b)\Bigr)^2,
\]
and Theorem~\ref{thm:free-b} evaluates this square. By the same interpolation argument as in Corollary~\ref{cor:Tmb-micro}, one also gets
\[
\sum_{r=0}^{n-1}\sum_{j=0}^{r} c_j(q,a;b)c_{r-j}(q,a;b)
\equiv
q^{4m(m+1)}\frac{(bq^{-n-1};q^2)_m^2}{(b;q^2)_m^2}
\pmod{(a-q^n)(1-aq^n)}.
\]
\end{remark}

We next record two useful collections of specializations of Corollary~\ref{cor:Tmb-micro}; their proofs are postponed until Section~\ref{sec:proof-cor45}. The first corresponds directly to the $\phi$- and $\omega$-type denominator patterns discussed above.

\begin{corollary}\label{cor:mock-specializations}
Let $n=2m+1$ be odd. Then modulo $(a-q^n)(1-aq^n)$, we have
\begin{align}
\sum_{j=0}^{m} q^{2j}\frac{(aq,q/a;q^2)_j}{(q^2;q^2)_j}
&\equiv q^{2m(m+1)}, \label{eq:b0-micro}\\
(1-q^n)\sum_{j=0}^{m} q^{2j}\frac{(aq,q/a;q^2)_j}{(q^2;q^2)_j(q;q^2)_{j+1}}
&\equiv (-1)^m q^{m^2+2m}, \label{eq:q3-micro}\\
\sum_{j=0}^{m} q^{2j}\frac{(aq,q/a;q^2)_j}{(q^2;q^2)_j(-q^2;q^2)_j}
&\equiv q^{m(m+1)}. \label{eq:bminusq2-micro}
\end{align}
Setting $a=1$ yields the corresponding congruences modulo $\Phi_n(q)^2$.
\end{corollary}

The next pair of specializations will be used for the nontrivial $p$-adic versions in Section~\ref{sec:twocolor}.

\begin{corollary}\label{cor:aux-specializations}
Let $n=2m+1$ be odd. Then modulo $(a-q^n)(1-aq^n)$, we have
\begin{align}
\sum_{j=0}^{m} q^{2j}\frac{(aq,q/a;q^2)_j}{(q^2;q^2)_j(q;q^2)_j}
&\equiv (-1)^m q^{m^2}\frac{1-q^n}{1-q},\label{eq:bq-micro}\\
\sum_{j=0}^{m} q^{2j}\frac{(aq,q/a;q^2)_j}{(q^2;q^2)_j(-q;q^2)_j}
&\equiv q^{m^2}\frac{1+q^n}{1+q}.\label{eq:bminusq-micro}
\end{align}
Setting $a=1$ yields the corresponding congruences modulo $\Phi_n(q)^2$.
\end{corollary}

In particular, Corollary~\ref{cor:mock-specializations} includes the shifted denominator $(q;q^2)_{n+1}$ via the specialization $b=q^3$, while the choice $b=q^{2k}$ recovers the two-color family
\[
\Fk(q)=\sum_{n\ge 0} q^{2n}\frac{(q;q^2)_n^2}{(q^2;q^2)_n(q^{2k};q^2)_n}.
\]
The case $k=1$ matches the prime-case congruence of Guo and Zeng~\cite[Theorem~1.6]{guo-zeng}, while $k\ge 2$ yields a family of zero congruences modulo $\Phi_n(q)^2$.

Section~\ref{sec:twocolor} records the two-color consequences and their $p$-adic versions. Sections~\ref{sec:proof-phi} and~\ref{sec:proof-nu} prove Theorem~\ref{thm:finite-theta} and Corollary~\ref{cor:microS}, and Theorem~\ref{thm:nu-finite-theta} and Corollary~\ref{cor:nu-micro}, respectively. Section~\ref{sec:proof-Tfamily} proves Theorem~\ref{thm:free-b} and Corollary~\ref{cor:Tmb-micro}, and Section~\ref{sec:proof-cor45} establishes Corollaries~\ref{cor:mock-specializations} and~\ref{cor:aux-specializations}. We conclude in Section~\ref{sec:outlook} with a broader conjectural problem for the remaining two-color families.

\section{Two-color consequences and \texorpdfstring{$p$}{p}-adic versions}\label{sec:twocolor}

We now record the two-color specialization of the denominator family and the resulting $p$-adic versions.

\begin{corollary}[Recovery of the $q^{2k}$-subfamily]\label{cor:Fk}
Let $k\in\mathbb{N}$ and let $n=2m+1$ be odd. Define
\begin{equation}\label{eq:Fmk}
\mathcal{F}_{m,k}(q,a):=\sum_{j=0}^{m} q^{2j}\frac{(aq,q/a;q^2)_j}{(q^2;q^2)_j(q^{2k};q^2)_j}.
\end{equation}
Then
\begin{equation}\label{eq:Fk-micro-new}
\mathcal{F}_{m,k}(q,a)
\equiv q^{2m(m+1)}\frac{(q^{2k-n-1};q^2)_m}{(q^{2k};q^2)_m}
\pmod{(a-q^n)(1-aq^n)}.
\end{equation}
If $\Phi_n(q)$ is coprime to $(q^{2k};q^2)_m$ (for instance, whenever $1\le k\le m+1$), then
\begin{equation}\label{eq:Fk-phi2-new}
\sum_{j=0}^{m} q^{2j}\frac{(q;q^2)_j^2}{(q^2;q^2)_j(q^{2k};q^2)_j}
\equiv q^{2m(m+1)}\frac{(q^{2k-n-1};q^2)_m}{(q^{2k};q^2)_m}
\pmod{\Phi_n(q)^2}.
\end{equation}
\end{corollary}

\begin{proof}
This is the specialization $b=q^{2k}$ of Corollary~\ref{cor:Tmb-micro}, followed by $a=1$ for the cyclotomic form.
\end{proof}

\begin{corollary}[$p$-adic specialization of the $q^{2k}$-subfamily]\label{cor:padic-k}
Let $p$ be an odd prime, $m=(p-1)/2$, and let $k\ge 1$ be an integer. Then
\begin{equation}\label{eq:padic}
\sum_{r=0}^{m}\frac{\binom{2r}{r}^2}{16^r\binom{k+r-1}{r}}
\equiv
\begin{cases}
(-1)^m & (k=1),\\
0 & (k\ge 2 \text{ and } p\ge 2k-1),
\end{cases}
\pmod{p^2}.
\end{equation}
\end{corollary}

\begin{proof}
As $q\to 1$, one has the standard asymptotics
\[
\frac{(q;q^2)_r}{(q^2;q^2)_r}\to \prod_{j=1}^{r}\frac{2j-1}{2j}=\frac{\binom{2r}{r}}{4^r},
\qquad
\frac{(q^2;q^2)_r}{(q^{2k};q^2)_r}\to \frac{1}{\binom{k+r-1}{r}}.
\]
Thus the left-hand side of~\eqref{eq:Fk-phi2-new} tends to the sum on the left-hand side of~\eqref{eq:padic}. Since $\Phi_p(1)=p$, the modulus $\Phi_p(q)^2$ becomes $p^2$ in the limit.

For $k=1$, the right-hand side of~\eqref{eq:Fk-phi2-new} tends to $(-1)^m$, yielding the first case of~\eqref{eq:padic}. For $k\ge 2$ and $p\ge 2k-1$, the factor $(q^{2k-p-1};q^2)_m$ contains $1-q^0$ and hence vanishes; the stated zero congruence follows.
\end{proof}

\begin{corollary}\label{cor:padic-q}
Let $p$ be an odd prime. Then
\begin{equation}\label{eq:padic-q}
\sum_{r=0}^{(p-1)/2}\frac{\binom{2r}{r}}{4^r}
\equiv (-1)^{(p-1)/2}p \pmod{p^2}.
\end{equation}
\end{corollary}

\begin{proof}
Let $m=(p-1)/2$. Specializing Corollary~\ref{cor:aux-specializations} to $a=1$ and $n=p$ gives
\[
\sum_{r=0}^{m} q^{2r}\frac{(q;q^2)_r}{(q^2;q^2)_r}
\equiv (-1)^m q^{m^2}\frac{1-q^p}{1-q}
\pmod{\Phi_p(q)^2}.
\]
Letting $q\to 1$, the left-hand side tends to the sum on the left-hand side of~\eqref{eq:padic-q}, while the right-hand side tends to $(-1)^m p$. As before, $\Phi_p(1)=p$, so the modulus $\Phi_p(q)^2$ becomes $p^2$.
\end{proof}

\begin{remark}\label{rem:nu-renormalized}
The nontrivial limits in Corollaries~\ref{cor:padic-k} and~\ref{cor:padic-q} may be viewed as renormalized $q\to 1$ shadows of the $\nu$-denominator finite-theta evaluation. Indeed, multiplying the $j$th summand of $N_m(q,a)$ by
\[
q^{-1}\frac{(-q;q^2)_{j+1}}{(q^2;q^2)_j^2}
\qquad\text{or}\qquad
q^{-1}\frac{(-q;q^2)_{j+1}}{(q^2;q^2)_j(q;q^2)_j}
\]
transforms it into the $k=1$ case of~\eqref{eq:Fmk} or the $b=q$ specialization of~\eqref{eq:Tmb}, respectively. In this sense the denominator family provides the natural normalization needed to see nontrivial $p$-adic shadows.
\end{remark}

\begin{corollary}\label{cor:padic-q3}
Let $p$ be an odd prime. Then
\begin{equation}\label{eq:padic-q3}
p\sum_{r=0}^{(p-1)/2}\frac{\binom{2r}{r}}{4^r(2r+1)}
\equiv (-1)^{(p-1)/2}\pmod{p^2}.
\end{equation}
\end{corollary}

\begin{proof}
Let $m=(p-1)/2$. Specializing Corollary~\ref{cor:mock-specializations} to $a=1$ and $n=p$ gives
\[
(1-q^p)\sum_{r=0}^{m} q^{2r}\frac{(q;q^2)_r^2}{(q^2;q^2)_r(q;q^2)_{r+1}}
\equiv (-1)^m q^{m^2+2m}
\pmod{\Phi_p(q)^2}.
\]
Since $(q;q^2)_{r+1}=(q;q^2)_r(1-q^{2r+1})$, the left-hand side equals
\[
(1-q^p)\sum_{r=0}^{m} q^{2r}\frac{(q;q^2)_r}{(q^2;q^2)_r(1-q^{2r+1})}.
\]
Now
\[
\frac{1-q^p}{1-q^{2r+1}}\to \frac{p}{2r+1},
\qquad
\frac{(q;q^2)_r}{(q^2;q^2)_r}\to \frac{\binom{2r}{r}}{4^r},
\]
so the left-hand side tends to the left-hand side of~\eqref{eq:padic-q3}, while the right-hand side tends to $(-1)^m$. As before, $\Phi_p(1)=p$, so the modulus $\Phi_p(q)^2$ becomes $p^2$.
\end{proof}

\begin{remark}[The case $k=1$ and Theorem~1.6 of Guo--Zeng]\label{rem:k1-gz}
Let $p$ be an odd prime and write $m=(p-1)/2$. Then~\eqref{eq:Fk-phi2-new} with $k=1$ gives
\[
\sum_{j=0}^{m} q^{2j}\frac{(q;q^2)_j^2}{(q^2;q^2)_j^2}
\equiv (-1)^m q^{m(m+1)}
\pmod{\Phi_p(q)^2}.
\]
Replacing $q$ by $q^{-1}$ and using
\[
(q^{-1};q^{-2})_j=(-1)^j q^{-j^2}(q;q^2)_j,
\qquad
(q^{-2};q^{-2})_j=(-1)^j q^{-j(j+1)}(q^2;q^2)_j,
\]
one finds that
\[
q^{-2j}\frac{(q^{-1};q^{-2})_j^2}{(q^{-2};q^{-2})_j^2}
=
\frac{(q;q^2)_j^2}{(q^2;q^2)_j^2}.
\]
Since $\Phi_p(q^{-1})$ differs from $\Phi_p(q)=[p]$ by a unit, the inverted form of the above congruence is exactly
\[
\sum_{j=0}^{m}\frac{(q;q^2)_j^2}{(q^2;q^2)_j^2}
\equiv (-1)^m q^{-m(m+1)}
\pmod{[p]^2},
\]
which is Theorem~1.6 of~\cite{guo-zeng} in the case $s=0$.
\end{remark}

\begin{remark}[Zero congruences in the $q^{2k}$-family]\label{rem:zero-family}
For fixed $k\ge 2$ and all odd $n\ge 2k-1$, the numerator $(q^{2k-n-1};q^2)_m$ contains the factor $1-q^0$, so the right-hand side of~\eqref{eq:Fk-phi2-new} vanishes identically. Thus for all such $n$ we obtain the zero congruence
\[
\sum_{j=0}^{(n-1)/2} q^{2j}\frac{(q;q^2)_j^2}{(q^2;q^2)_j(q^{2k};q^2)_j}
\equiv 0 \pmod{\Phi_n(q)^2}.
\]
This vanishing phenomenon is reminiscent of the zero-type families proved by Guo and Zeng in Theorem~1.3 and Corollary~1.4 of~\cite{guo-zeng}. Finally, letting $k\to\infty$ in~\eqref{eq:Fk-micro-new} recovers the base companion congruence~\eqref{eq:b0-micro}.
\end{remark}

\section{Proof of Theorem 1 and Corollary~1}\label{sec:proof-phi}

\begin{proof}[Proof of Theorem~\ref{thm:finite-theta}]
Following the recurrence-comparison viewpoint used by Guo and Liu in their proof of Lemma~2.2 of~\cite{guo-liu}, we show that the two sides of~\eqref{eq:finite-theta} satisfy the same recurrence and initial value.

Write
\begin{align*}
L_m(q)&:=\sum_{j=0}^{m} q^{2j+1}\frac{(q^{-2m},q^{2m+2};q^2)_j}{(-q^2;q^2)_j},\\
R_m(q)&:=(-1)^m q^{2m(m+1)+1}\sum_{r=-m}^{m}(-1)^r q^{-2r^2},
\end{align*}
so that~\eqref{eq:finite-theta} is the identity $L_m(q)=R_m(q)$.

For $m\ge 1$ and $0\le j\le m$, set
\begin{align*}
F_{m,j}(q)&:=q^{2j+1}\frac{(q^{-2m},q^{2m+2};q^2)_j}{(-q^2;q^2)_j},\\
G_{m,j}(q)&:=q^{2m+1}(1+q^{2j})\frac{(q^{-2m},q^{2m};q^2)_j}{(-q^2;q^2)_j}.
\end{align*}
We claim that
\begin{equation}\label{eq:FG-telescope}
F_{m,j}(q)+q^{4m}F_{m-1,j}(q)=G_{m,j}(q)-G_{m,j+1}(q).
\end{equation}
Indeed,
\begin{align*}
F_{m,j}(q)+q^{4m}F_{m-1,j}(q)
&=
q^{2j+1}\frac{(q^{-2m},q^{2m};q^2)_j}{(-q^2;q^2)_j}\\
&\qquad\times
\left(
\frac{1-q^{2m+2j}}{1-q^{2m}}
+
q^{4m}\frac{1-q^{-2m+2j}}{1-q^{-2m}}
\right)\\
&=
q^{2j+1}\frac{(q^{-2m},q^{2m};q^2)_j}{(-q^2;q^2)_j}
\bigl(q^{4m}+q^{2m}+1-q^{2m+2j}\bigr),
\end{align*}
while
\begin{align*}
G_{m,j}(q)-G_{m,j+1}(q)
&=
q^{2m+1}\frac{(q^{-2m},q^{2m};q^2)_j}{(-q^2;q^2)_j}\\
&\qquad\times\Bigl((1+q^{2j})-(1-q^{-2m+2j})(1-q^{2m+2j})\Bigr)\\
&=
q^{2j+1}\frac{(q^{-2m},q^{2m};q^2)_j}{(-q^2;q^2)_j}
\bigl(q^{4m}+q^{2m}+1-q^{2m+2j}\bigr).
\end{align*}
This proves~\eqref{eq:FG-telescope}.

Note that $F_{m-1,m}(q)=0$ and $G_{m,m+1}(q)=0$. Summing~\eqref{eq:FG-telescope} over $j=0,\dots,m$, we find
\[
L_m(q)+q^{4m}L_{m-1}(q)
=
\sum_{j=0}^{m}\bigl(G_{m,j}(q)-G_{m,j+1}(q)\bigr)
=
G_{m,0}(q)
=
2q^{2m+1}.
\]
Hence
\begin{equation}\label{eq:L-recurrence}
L_m(q)=-q^{4m}L_{m-1}(q)+2q^{2m+1},
\qquad L_0(q)=q.
\end{equation}

Now put
\[
\Theta_m(q):=\sum_{r=-m}^{m}(-1)^r q^{-2r^2}.
\]
Then
\[
\Theta_m(q)-\Theta_{m-1}(q)=2(-1)^m q^{-2m^2},
\]
and therefore
\[
R_m(q)+q^{4m}R_{m-1}(q)
=
(-1)^m q^{2m(m+1)+1}\bigl(\Theta_m(q)-\Theta_{m-1}(q)\bigr)
=
2q^{2m+1}.
\]
Together with $R_0(q)=q$, this shows that $R_m(q)$ satisfies the same recurrence~\eqref{eq:L-recurrence} and the same initial value as $L_m(q)$. Thus $L_m(q)=R_m(q)$ for all $m\ge 0$.
\end{proof}

\begin{proof}[Proof of Corollary~\ref{cor:microS}]
Let
\begin{align*}
F_m(a)&:=\sum_{j=0}^{m}\frac{q^{2j+1}(aq,q/a;q^2)_j}{(-q^2;q^2)_j},\\
\Theta_m(q)&:=(-1)^m q^{2m(m+1)+1}\sum_{r=-m}^{m}(-1)^r q^{-2r^2}.
\end{align*}
By Theorem~\ref{thm:finite-theta},
\[
F_m(q^n)=\Theta_m(q).
\]
Since the summand is invariant under $a\mapsto a^{-1}$, we also have
\[
F_m(q^{-n})=F_m(q^n)=\Theta_m(q).
\]
Now $F_m(a)$ is a Laurent polynomial in $a$, so
\[
P_m(a):=a^m\bigl(F_m(a)-\Theta_m(q)\bigr)\in \Q(q)[a].
\]
Because $P_m(q^n)=P_m(q^{-n})=0$, the polynomial $P_m(a)$ is divisible by
\[
(a-q^n)(a-q^{-n})
\]
in $\Q(q)[a]$. Since $a^m$ is a unit in $\Q(q)[a,a^{-1}]$, it follows that
\[
F_m(a)-\Theta_m(q)\equiv 0 \pmod{(a-q^n)(a-q^{-n})}.
\]
Finally,
\[
(a-q^n)(a-q^{-n})=-q^{-n}(a-q^n)(1-aq^n),
\]
so the two moduli differ by a unit in $\Q(q)$. This proves~\eqref{eq:micro-S}. Specializing~\eqref{eq:micro-S} at $a=q^n$ recovers Theorem~\ref{thm:finite-theta}, so the two statements are equivalent.
\end{proof}

\section{Proof of Theorem 2 and Corollary~2}\label{sec:proof-nu}

\begin{proof}[Proof of Theorem~\ref{thm:nu-finite-theta}]
Write
\begin{align*}
L_m(q)&:=(1+q^{2m+1})\sum_{j=0}^{m} q^{2j+1}
\frac{(q^{-2m},q^{2m+2};q^2)_j}{(-q;q^2)_{j+1}},\\
R_m(q)&:=(-1)^m q^{(m+1)(2m+1)}
\sum_{r=-m}^{m} (-1)^r q^{-2r^2-r}.
\end{align*}
Then~\eqref{eq:nu-finite-theta} is exactly the identity $L_m(q)=R_m(q)$.

For $m\ge 1$ and $0\le j\le m$, define
\begin{align*}
F_{m,j}(q)&:=(1+q^{2m+1})q^{2j+1}
\frac{(q^{-2m},q^{2m+2};q^2)_j}{(-q;q^2)_{j+1}},\\
G_{m,j}(q)&:=q^{2m+1}(1+q^{2m})
\frac{(q^{-2m},q^{2m};q^2)_j}{(-q;q^2)_j}.
\end{align*}
We claim that
\begin{equation}\label{eq:nu-telescope}
F_{m,j}(q)+q^{4m+1}F_{m-1,j}(q)=G_{m,j}(q)-G_{m,j+1}(q).
\end{equation}
Indeed,
\begin{align*}
F_{m,j}(q)+q^{4m+1}F_{m-1,j}(q)
&=
q^{2j+1}\frac{(q^{-2m},q^{2m};q^2)_j}{(-q;q^2)_{j+1}}\\
&\quad\times\Biggl(
\frac{(1+q^{2m+1})(1-q^{2m+2j})}{1-q^{2m}}\\
&\qquad\qquad
+\frac{q^{4m+1}(1+q^{2m-1})(1-q^{-2m+2j})}{1-q^{-2m}}
\Biggr).
\end{align*}
Since
\[
\frac{1}{1-q^{-2m}}=-\frac{q^{2m}}{1-q^{2m}},
\]
the factor in parentheses equals
\[
\begin{aligned}
\frac{(1+q^{2m+1})(1-q^{2m+2j})}{1-q^{2m}}
&- \frac{q^{6m+1}(1+q^{2m-1})(1-q^{-2m+2j})}{1-q^{2m}}\\
&= (1+q^{2m})\bigl(1+q^{2m+1}+q^{4m}-q^{2m+2j}\bigr).
\end{aligned}
\]
and therefore
\begin{equation}\label{eq:nu-left-common}
\begin{aligned}
F_{m,j}(q)+q^{4m+1}F_{m-1,j}(q)
&=
q^{2j+1}\frac{(q^{-2m},q^{2m};q^2)_j}{(-q;q^2)_{j+1}}\\
&\quad\times(1+q^{2m})\bigl(1+q^{2m+1}+q^{4m}-q^{2m+2j}\bigr).
\end{aligned}
\end{equation}

On the other hand,
\begin{align*}
G_{m,j}(q)-G_{m,j+1}(q)
&=
q^{2m+1}(1+q^{2m})
\frac{(q^{-2m},q^{2m};q^2)_j}{(-q;q^2)_j}\\
&\quad\times
\left(
1-\frac{(1-q^{-2m+2j})(1-q^{2m+2j})}{1+q^{2j+1}}
\right).
\end{align*}
Using $(-q;q^2)_{j+1}=(-q;q^2)_j(1+q^{2j+1})$, this becomes
\[
q^{2m+1}(1+q^{2m})
\frac{(q^{-2m},q^{2m};q^2)_j}{(-q;q^2)_{j+1}}
\Bigl((1+q^{2j+1})-(1-q^{-2m+2j})(1-q^{2m+2j})\Bigr).
\]
Now
\[
(1+q^{2j+1})-(1-q^{-2m+2j})(1-q^{2m+2j})
=
q^{-2m+2j}(1+q^{2m+1}+q^{4m}-q^{2m+2j}),
\]
so
\begin{equation}\label{eq:nu-right-common}
G_{m,j}(q)-G_{m,j+1}(q)
=
q^{2j+1}\frac{(q^{-2m},q^{2m};q^2)_j}{(-q;q^2)_{j+1}}
(1+q^{2m})(1+q^{2m+1}+q^{4m}-q^{2m+2j}).
\end{equation}
Comparing~\eqref{eq:nu-left-common} and~\eqref{eq:nu-right-common} proves~\eqref{eq:nu-telescope}.

Now $F_{m-1,m}(q)=0$ and $G_{m,m+1}(q)=0$. Summing~\eqref{eq:nu-telescope} over $j=0,\dots,m$, we obtain
\[
L_m(q)+q^{4m+1}L_{m-1}(q)
=
\sum_{j=0}^{m}\bigl(G_{m,j}(q)-G_{m,j+1}(q)\bigr)
=
G_{m,0}(q)
=
q^{2m+1}(1+q^{2m}).
\]
Hence
\begin{equation}\label{eq:nu-L-recurrence}
L_m(q)=-q^{4m+1}L_{m-1}(q)+q^{2m+1}(1+q^{2m}),
\qquad
L_0(q)=q.
\end{equation}

Now set
\[
\Psi_m(q):=\sum_{r=-m}^{m}(-1)^r q^{-2r^2-r}.
\]
Then
\[
\Psi_m(q)-\Psi_{m-1}(q)
=
(-1)^m\bigl(q^{-2m^2-m}+q^{-2m^2+m}\bigr)
=
(-1)^m q^{-2m^2-m}(1+q^{2m}).
\]
Therefore
\begin{align*}
R_m(q)+q^{4m+1}R_{m-1}(q)
&=
(-1)^m q^{(m+1)(2m+1)}\bigl(\Psi_m(q)-\Psi_{m-1}(q)\bigr)\\
&=
q^{2m+1}(1+q^{2m}).
\end{align*}
Together with $R_0(q)=q$, this shows that $R_m(q)$ satisfies the same recurrence~\eqref{eq:nu-L-recurrence} and the same initial value as $L_m(q)$. Thus $L_m(q)=R_m(q)$ for all $m\ge 0$.
\end{proof}

\begin{proof}[Proof of Corollary~\ref{cor:nu-micro}]
Let
\[
F_m(a):=(1+q^n)\sum_{j=0}^{m} q^{2j+1}\frac{(aq,q/a;q^2)_j}{(-q;q^2)_{j+1}}
\]
and
\[
\Theta_m(q):=
(-1)^m q^{(m+1)(2m+1)}
\sum_{r=-m}^{m} (-1)^r q^{-2r^2-r}.
\]
By Theorem~\ref{thm:nu-finite-theta},
\[
F_m(q^n)=\Theta_m(q).
\]
Since the summand is invariant under $a\mapsto a^{-1}$, we also have
\[
F_m(q^{-n})=F_m(q^n)=\Theta_m(q).
\]
Now $F_m(a)$ is a Laurent polynomial in $a$, so
\[
P_m(a):=a^m\bigl(F_m(a)-\Theta_m(q)\bigr)\in \Q(q)[a].
\]
Because $P_m(q^n)=P_m(q^{-n})=0$, the polynomial $P_m(a)$ is divisible by
\[
(a-q^n)(a-q^{-n}).
\]
Since
\[
(a-q^n)(a-q^{-n})=-q^{-n}(a-q^n)(1-aq^n),
\]
the two moduli differ by a unit in $\Q(q)$, and the stated congruence follows.
\end{proof}

\begin{remark}[Harmless $q\to 1$ limits]\label{rem:harmless-limits}
Letting $q\to 1$ in either Theorem~\ref{thm:finite-theta} or Theorem~\ref{thm:nu-finite-theta} yields only the trivial identity $1=1$. Thus the nontrivial $p$-adic shadows in this paper come not from the direct limits of the finite-theta identities, but from Theorem~\ref{thm:free-b}, Corollaries~\ref{cor:mock-specializations}--\ref{cor:aux-specializations}, and the two-color specializations recorded in Section~\ref{sec:twocolor}.
\end{remark}

\section{Proof of Theorem 3 and Corollary 3}\label{sec:proof-Tfamily}

We now prove Theorem~\ref{thm:free-b} and Corollary~\ref{cor:Tmb-micro}, which underlie Corollaries~\ref{cor:mock-specializations} and~\ref{cor:aux-specializations}.

\begin{proof}[Proof of Theorem~\ref{thm:free-b}]
Set $Q=q^2$. Then
\[
\T_m(q,q^n;b)
=\sum_{j=0}^{m}\frac{(Q^{-m},Q^{m+1};Q)_j}{(Q,b;Q)_j}Q^j
={}_2\phi_1\!\left(\begin{matrix}Q^{-m},\,Q^{m+1}\\b\end{matrix};\,Q,Q\right).
\]
A classical terminating ${}_2\phi_1$ summation (see, for example,~\cite[eq.~(1.5.1)]{gasper-rahman}) gives
\[
{}_2\phi_1\!\left(\begin{matrix}Q^{-m},\,A\\C\end{matrix};\,Q,Q\right)
=A^m\frac{(C/A;Q)_m}{(C;Q)_m}
\qquad (m\in\mathbb{N}).
\]
Applying this with $A=Q^{m+1}$ and $C=b$ yields
\[
\T_m(q,q^n;b)
=Q^{m(m+1)}\frac{(bQ^{-m-1};Q)_m}{(b;Q)_m}
=q^{2m(m+1)}\frac{(bq^{-n-1};q^2)_m}{(b;q^2)_m}.
\]
The case $a=q^{-n}$ is identical, since the summand in~\eqref{eq:Tmb} is invariant under $a\mapsto a^{-1}$.
\end{proof}

\begin{proof}[Proof of Corollary~\ref{cor:Tmb-micro}]
By Theorem~\ref{thm:free-b}, the difference in~\eqref{eq:Tmb-micro} vanishes at $a=q^n$ and $a=q^{-n}$. For fixed $q$ and $b$, it is a Laurent polynomial in $a$, and multiplying by $a^m$ clears negative powers. Symmetry under $a\mapsto a^{-1}$ then shows divisibility by $(a-q^n)(a-q^{-n})$. This differs from $(a-q^n)(1-aq^n)$ only by the unit $-q^{-n}$.
\end{proof}

\section{Proofs for Corollaries 4-5}\label{sec:proof-cor45}

\begin{proof}[Proof of Corollary~\ref{cor:mock-specializations}]
The three congruences follow from Corollary~\ref{cor:Tmb-micro} by the specializations $b=0$, $q^3$, and $-q^2$, respectively.

For $b=0$ we use $(0;q^2)_j=1$ and obtain~\eqref{eq:b0-micro}. Setting $a=1$ yields
\begin{equation}\label{eq:b0-phi2}
\sum_{j=0}^{m} q^{2j}\frac{(q;q^2)_j^2}{(q^2;q^2)_j}
\equiv q^{2m(m+1)}
\pmod{\Phi_n(q)^2}.
\end{equation}

For $b=q^3$ we use
\[
(q;q^2)_{j+1}=(1-q)(q^3;q^2)_j
\]
and
\[
\frac{(q^{2-n};q^2)_m}{(q^3;q^2)_m}
=(-1)^m q^{-m^2}\frac{1-q}{1-q^n}.
\]
Hence Theorem~\ref{thm:free-b} gives the exact evaluation
\begin{equation}\label{eq:q3-exact-balanced}
(1-q^n)\sum_{j=0}^{m} q^{2j}\frac{(q^{1-n},q^{n+1};q^2)_j}{(q^2;q^2)_j(q;q^2)_{j+1}}
=
(-1)^m q^{m^2+2m},
\end{equation}
and Corollary~\ref{cor:Tmb-micro} yields~\eqref{eq:q3-micro}. Setting $a=1$ gives
\begin{equation}\label{eq:q3-phi2}
(1-q^n)\sum_{j=0}^{m} q^{2j}\frac{(q;q^2)_j^2}{(q^2;q^2)_j(q;q^2)_{j+1}}
\equiv (-1)^m q^{m^2+2m}
\pmod{\Phi_n(q)^2}.
\end{equation}

For $b=-q^2$ we use
\[
q^{2m(m+1)}\frac{(-q^{1-n};q^2)_m}{(-q^2;q^2)_m}=q^{m(m+1)}
\]
and obtain~\eqref{eq:bminusq2-micro}. Setting $a=1$ gives
\begin{equation}\label{eq:bminusq2-phi2}
\sum_{j=0}^{m} q^{2j}\frac{(q;q^2)_j^2}{(q^2;q^2)_j(-q^2;q^2)_j}
\equiv q^{m(m+1)}
\pmod{\Phi_n(q)^2}.
\end{equation}
\end{proof}

\begin{proof}[Proof of Corollary~\ref{cor:aux-specializations}]
These are the specializations $b=q$ and $b=-q$ of Corollary~\ref{cor:Tmb-micro}. We use
\[
\frac{(q^{-n};q^2)_m}{(q;q^2)_m}
=(-1)^m q^{-m^2-2m}\frac{1-q^n}{1-q},
\qquad
\frac{(-q^{-n};q^2)_m}{(-q;q^2)_m}
=q^{-m^2-2m}\frac{1+q^n}{1+q}.
\]
Setting $a=1$ yields
\begin{align}
\sum_{j=0}^{m} q^{2j}\frac{(q;q^2)_j}{(q^2;q^2)_j}
&\equiv (-1)^m q^{m^2}\frac{1-q^n}{1-q}
\pmod{\Phi_n(q)^2},\label{eq:bq-phi2}\\
\sum_{j=0}^{m} q^{2j}\frac{(q;q^2)_j^2}{(q^2;q^2)_j(-q;q^2)_j}
&\equiv q^{m^2}\frac{1+q^n}{1+q}
\pmod{\Phi_n(q)^2}.\label{eq:bminusq-phi2}
\end{align}
Here $(q;q^2)_m$ and $(-q;q^2)_m$ are coprime to $\Phi_n(q)$ because $n$ is odd.
\end{proof}

\section{Outlook on the two-color sequences}\label{sec:outlook}

We conclude by returning to the broader two-color families introduced in~\cite{andrews-elbachraoui}:
\begin{align}
\Ck(q)&:=\sum_{n\ge 0} q^{2n}\frac{(q;q^2)_n^2}{(-q^2;q^2)_n(-q^{2k};q^2)_n},\label{eq:Ck}\\
\Gk(q)&:=\sum_{n\ge 0} q^{2n}\frac{(q;q^2)_n(-q;q^2)_n}{(q^2;q^2)_n(-q^{2k};q^2)_n},\label{eq:Gk}\\
\Hk(q)&:=\sum_{n\ge 0} q^{2n}\frac{(q;q^2)_n(-q;q^2)_n}{(q^2;q^2)_n(q^{2k};q^2)_n}.\label{eq:Hk}
\end{align}
As $k\to\infty$, the family~\eqref{eq:Ck} tends to
\[
\sum_{n\ge 0} q^{2n}\frac{(q;q^2)_n^2}{(-q^2;q^2)_n},
\]
which differs from~\eqref{eq:S} only by the initial factor $q$. It is therefore natural to ask for microscope congruences for truncated versions of~\eqref{eq:Ck}--\eqref{eq:Hk}, ideally with finite theta evaluations at $a=q^{\pm n}$ analogous to Theorems~\ref{thm:finite-theta},~\ref{thm:nu-finite-theta}, and~\ref{thm:free-b}.

\begin{conjecture}\label{conj:broader}
Fix $k\ge 1$ and let $n=2m+1$ be odd. For each of the series in~\eqref{eq:Ck}--\eqref{eq:Hk}, define a natural $a$-extension by replacing $(q;q^2)_j^2$ with $(aq,q/a;q^2)_j$ and $(q;q^2)_j(-q;q^2)_j$ with a suitable two-parameter analogue. Then the corresponding truncations at $j=m$ should satisfy microscopic congruences modulo $(a-q^n)(1-aq^n)$ whose specializations at $a=q^{\pm n}$ are expressible in terms of finite theta polynomials or short linear combinations thereof.
\end{conjecture}

\begin{remark}
At present we do not have closed forms analogous to~\eqref{eq:Tmb-exact} for the $(-q^2;q^2)_j$ denominator families in~\eqref{eq:Ck} and~\eqref{eq:Gk}. The proved behavior of the $\phi$-type and $\nu$-type theorems suggests that such formulas, if they exist, may involve finite theta sums rather than a single $q$-Pochhammer quotient. This is the main structural reason for separating the paper into the finite-theta part of Sections~\ref{sec:proof-phi} and~\ref{sec:proof-nu} and the exactly summable denominator family based on Theorem~\ref{thm:free-b}, Corollaries~\ref{cor:mock-specializations}--\ref{cor:aux-specializations}, and Section~\ref{sec:twocolor}. More conceptually, Conjecture~\ref{conj:broader} should be viewed as part of a broader bridge between mock theta functions and creative microscoping. Ramanujan’s original definition is local at roots of unity, and creative microscoping is local at roots of unity as well. The finite theta polynomials appearing in this paper may therefore be viewed as local theta corrections for truncated mock-theta-type sums. From that viewpoint, it would be more surprising for the remaining two-color families to lie permanently outside the scope of a mock-theta $q$-microscope than for them to admit more complicated finite theta descriptions.
\end{remark}

\medskip
\noindent\textit{Data Availability Statement.} No data was used for the research described in the article.

\end{document}